\documentclass[12pt]{article}
\usepackage{amsthm, amsmath, amssymb, enumerate}
\usepackage{fullpage}
\usepackage[colorlinks=true]{hyperref}

\theoremstyle{plain}
\newtheorem{thm}{Theorem}[section]
\newtheorem{theorem}[thm]{Theorem}
\newtheorem{corollary}[thm]{Corollary}
\newtheorem{lemma}[thm]{Lemma}
\newtheorem{proposition}[thm]{Proposition}

\theoremstyle{remark}
\newtheorem{remark}[thm]{Remark}

\numberwithin{equation}{subsection}

\newcommand{\ZZ}{\mathbb{Z}}
\newcommand{\CC}{\mathbb{C}}
\newcommand{\QQ}{\mathbb{Q}}
\newcommand{\RR}{\mathbb{R}}
\newcommand{\PP}{\mathbb{P}}

\newcommand{\CA}{{\mathcal {A}}}

\newcommand{\CD}{{\mathcal{D}}}

\newcommand{\CL}{{\mathcal {L}}}
\newcommand{\CM}{{\mathcal {M}}}

\newcommand{\CO}{{\mathcal {O}}}

\newcommand{\CX}{{\mathcal {X}}}
\newcommand{\CY}{{\mathcal {Y}}}
\newcommand{\CZ}{{\mathcal {Z}}}

\newcommand{\OL}{{\overline{L}}}

\newcommand{\OD}{{\overline{D}}}

\renewcommand{\d}{\textnormal{d}}
\newcommand{\dc}{\textnormal{d}^{\textnormal{c}}}

\newcommand{\e}{\mathrm{e}}

\newcommand{\an}{{\mathrm{an}}}

\newcommand{\Div}{{\mathrm{Div}}}
\renewcommand{\div}{{\mathrm{div}}}

\newcommand{\Pic}{\mathrm{Pic}}

\newcommand\HDiv{\widehat{\mathrm{Div}}}
\newcommand\HPic{\widehat{\mathrm{Pic}}}

\newcommand{\sm}{{\mathrm{sm}}}

\newcommand{\Spec}{\mathrm{Spec}\,}

\newcommand\TL{\widetilde{L}}

\newcommand\hdiv{\widehat{\mathrm{div}}}

\newcommand{\supp}{{\mathrm{supp}}}
\newcommand\Gm{\mathbb{G}_{\mathrm{m}}}

\newcommand\intt{\mathrm{int}}
\newcommand\TPic{\widetilde{\mathrm{Pic}}}

\newcommand\snef{\mathrm{snef}}
\newcommand\trop{\mathrm{trop}}

\begin{document}
	\title{An integration formula of Chern forms on quasi-projective varieties}
	\author{Ruoyi Guo}
	\maketitle
	\tableofcontents

\section{Introduction}

In complex geometry, a famous formula of Shiing Shen Chern is as follows.
   
\begin{theorem}\label{integration formula in proj complex case}
	Let $X$ be a projective variety over $\CC$ of dimension $n$. Let $\OL_1,\cdots,\OL_n$ be hermitian line bundles on $X$ with underlying line bundles $L_1,\cdots, L_n$. Suppose the metric of each $\OL_i$ is smooth. Then 
	\[
	\int_{X(\CC)}c_1(\OL_1)\cdots c_1(\OL_n)=L_1\cdots L_n.
	\]
	Here the left hand side is the integration of Chern forms; the right hand side is the algebraic intersection number.
\end{theorem}

The formula is generalized to the non-archimedean case, i.e. $X$ is a projective variety over a complete non-archimedean valuation field $K$ instead of $\CC$. In \cite{CL06}, Chambert-Loir defines the \emph{Chambert-Loir measure} $c_1(\OL_1)\cdots c_1(\OL_n)$ on $X^\an$ for integrable metrized line bundles $\OL_1,\cdots,\OL_n$ in the sense of \cite{Zha95}. Then he proves a non-archimedean version of Theorem \ref{integration formula in proj complex case} in \cite{CL06}. After that, Gubler generalizes it to formally metrized line bundles in \cite{Gub07}. In \cite{CLD12}, Chambert-Loir and Ducros develop a deep theory of real valued $(p,q)$-forms on $X^\an$. By the powerful theory, they redefine the measure $c_1(\OL_1)\cdots c_1(\OL_n)$ in a more analytic way and deduce the formula above in this analytic setting. In \cite{YZ21}, Yuan and  Zhang generalize the theory in \cite{Zha95} and define \emph{adelic line bundles} on a quasi-projective variety $U$ as an appropriate limit of line bundles on compactifications of $U$. In the loc. cit., they generalize the measure and the intersection theory to adelic line bundles on a quasi-projective variety.

The main purpose of this paper is to generalize Theorem \ref{integration formula in proj complex case} to the quasi-projective case for both archimedean and non-archimedean fields in the setting of adelic line bundles of \cite{YZ21}. In this paper, we call a field $K$ a \emph{complete valuation field}, if $K$ is complete with respect to a non-trivial absolute value. If $K$ is archimedean, let $O_K=K$; if $K$ is non-archimedean, let $O_K$ be the valuation ring of $K$. 

Let $U$ be a quasi-projective variety over a complete valuation field $K$. In \cite[Section 2.7]{YZ21}, authors define a group $\HPic(U/O_K)_\intt$ of \emph{integrable adelic line bundles}. They also define a group $\TPic(U/K)_\intt$ of \emph{integrable adelic line bundles without metrics} on $U$. There is a natural forgetful homomorphism 
\[
\HPic(U/O_K)_\intt\to\TPic(U/K)_\intt,
\]
which is defined in \cite[Section 3.6]{YZ21}. This homomorphism is to forget the metric information of an adelic line bundle. Yuan and Zhang define a measure $c_1(\OL_1)\cdots c_1(\OL_n)$ on $U^\an$ for $\OL_1,\cdots,\OL_n\in\HPic(U/O_K)_\intt$ in \cite[Section 3.6]{YZ21} which generalizes the Monge-Ampère measure in the complex case and the Chambert-Loir measure in the non-archimedean case. They also define an intersection number $\TL_1\cdots \TL_n$ for $\TL_1,\cdots,\TL_n\in\TPic(U/K)_\intt$ in \cite[Section 4.1]{YZ21} generalizing the usual intersection number which only makes sense on projective varieties. Here we use the notation $\TPic(U/K)$ instead of the notation $\HPic(U/K)$ in the loc. cit. We will illustrate the notions later in Section \ref{section of adelic}.

We state the main theorem as follows.
	
\begin{theorem}[main theorem]\label{main theorem}
	Let $U$ be a quasi-projective variety of dimension $n$ over a complete valuation field $K$. Let $U^\an$ be the Berkovich space of $U$ over $K$. Then for integrable adelic line bundles $\OL_1,\cdots,\OL_n\in\HPic(U/O_K)_\intt$, it holds that
	\[
	\int_{U^\an}c_1(\OL_1)\cdots c_1(\OL_n)=\TL_1\cdots \TL_n.
	\]
	Here $\TL_i$ is the image of $\OL_i$ under the natural homomorphism $\HPic(U/O_K)_\intt\to\TPic(U/K)_\intt$.
\end{theorem} 

In \cite{GV19}, Gauthier and Vigny prove a version of Theorem \ref{main theorem} in complex dynamics setting. In \cite[Lemma 5.4.4]{YZ21}, they prove Theorem \ref{main theorem} in the case that each adelic line bundle $\OL_i$ is a localization of a global integrable adelic line bundle over a global field. To be more specific, there are integrable adelic line bundles $\OL_1',\cdots,\OL_n'$ on a quasi-projective variety $U'$ over a number field or a function field $F$ such that for some place $v$ of $F$, we have $U'\otimes_F F_v\cong U$ and $\OL_i'\otimes_F F_v\cong \OL_i$ for all $1\leq i\leq n$. The key of the proof of this case in loc. cit. is a formula of global arithmetic intersection number which gives an upper bound of local integrations. Nevertheless, our proof in this paper is purely local. It does not use any global data in the proof.

Our theorem has a close relation with the notion of \emph{non-degeneracy} used in the proof of the uniform Mordell-Lang in \cite{DGH20}. Let $S$ be a quasi-projective variety over $\CC$. Let $(X,f,L)$ be a polarized dynamical system over $S$, i.e. $X$ is an integral scheme projective and flat over $S$; $f:X\to X$ is a morphism over $S$; $L\in\Pic(X)_\QQ$ is a $\QQ$-line bundle over $X$, relatively ample over $S$ such that $f^*L=qL$ for some rational number $q>1$. Then by \cite[Theorem 6.1.1]{YZ21}, these data determine a unique nef adelic line bundle $\OL_f$ on $X$ such that $f^*\OL_f=q\OL_f$. Let $\TL_f$ be the image of $\OL_f$ under the forgetful homomorphism 
\[
\HPic(U/\CC)_\intt\to\TPic(U/\CC)_\intt.
\]
For a closed subvariety $Y$ of $X$, we say that $Y$ is \emph{non-degenerate}, if $(\TL_f|_Y)^{\dim Y}>0$. By our theorem,
\[
(\TL_f|_Y)^{\dim Y}=\int_{Y^\an}c_1(\OL_f|_Y)^{\dim Y}.
\]
Since the measure on the right hand side is already positive, $Y$ is non-degenerate if and only if the measure $c_1(\OL_f|_Y)^{\dim Y}$ is non-zero. In particular, when $X\to S$ is an abelian scheme over a smooth quasi-projective variety, the measure coincides with the wedge product of the Betti form in \cite{DGH20,CGHX21}. The main difficulty for the general setting $(X,f,L)$ is to define such a measure, which is defined by \cite{YZ21}, and to relate the measure with the algebraic intersection number. Our main theorem finds such a relation. In addition, it allows us to consider the similar notions over a non-archimedean field.

In a recent work \cite{XY23} of Xie and Yuan, they introduce a new approach to the geometric Bombieri–Lang conjecture. During their proof, they consider an abelian scheme $\CA\to B$ with a rigidified, relatively ample and symmetric line bundle $\CL$ on $\CA$, where $B$ is a smooth quasi-projective curve. In \cite[Theorem 3.1]{XY23}, they use the formula
\[
\hat{h}_{\CL}(s)=\int_{B}s^*\omega.
\]
Here $\hat{h}_{\CL}$ is the N\'eron-Tate height induced by $\CL$; $s$ is a section $s:B\to\CA$; $\omega$ is the Betti form of $\CL$ on $\CA$. In fact, by Tate's limit argument and the section $s$, we get an adelic line bundle $\OL$ on $B$ and an adelic line bundle $\TL$ without metric induced by $\OL$ , such that 
\[
\hat{h}_{\CL}(s)=\deg\TL,\quad  \int_{B}s^*\omega=\int_B c_1(\OL).
\]
Thus this formula is indeed a one-dimensional special case of our main theorem. Note that this special case is first proven by \cite{GV19} since an abelian scheme is automatically a dynamical system. However, it is still a crucial special case of our main theorem.

As the last part of this section, we explain the main idea of the proof. The proof is mainly based on limit processes. In \cite{YZ21}, Yuan and Zhang define adelic line bundles on $U$ as a limit of (metrized) line bundles on compactifications of $U$ satisfying a Cauchy condition. By the Cauchy condition, we can deal with the intersection number $\TL_1\cdots \TL_n$ as a limit of the usual intersection numbers on projective varieties. The hard part is the integration on the left hand side. In the loc. cit., they define the measure $c_1(\OL_1)\cdots c_1(\OL_n)$ by a weak convergence process on $U^\an$. The measure is a weak limit of measures on $U^\an$ induced by line bundles on compactifications of $U$. As a result, we cannot compute the the integration 
\[
\int_{U^\an}c_1(\OL_1)\cdots c_1(\OL_n),
\]
directly since the function $1$ is not compactly supported on $U^\an$. To compute the integration, we will use some cut-off functions which have compact supports on $U^\an$ and take a limit. 
	
In summary, there are two limit processes in the proof. One is the limit process from the compactifications which defines the measure. The other is the Lebesgue monotone convergence theorem which helps us compute the integration.
	
A key of the proof is to find proper cut-off functions to harmonize the two limit processes. The cut-off functions that we use here as a powerful tool are called \emph{dsh} functions, which is the short for \emph{difference of pluri-subharmonic} functions. They were introduced by \cite{DS09}, notably for applications in complex dynamics, and used as test functions to compute volumes in \cite{GOV19,GV19}.

\subsubsection*{Acknowledgment}

I would like to express my sincere thanks to my advisor, Professor Xinyi Yuan, for his invaluable advice and persistent encouragement. I would also thank Thomas Gauthier and Gabriel Vigny since this paper is inspired by their papers. I am grateful to Yulin Cai, Walter Gubler, Yinchong Song, Gabriel Vigny and the anonymous referee for reviewing an early version of this paper and providing some valuable suggestions. Finally, I would thank my classmates Lai Shang and Jiarui Song for some helpful discussions.

\section{Adelic divisors and adelic line bundles \\on quasi-projective varieties} \label{section of adelic}

In this subsection, we recall the theory of adelic line bundles. Our definitions here can be found in \cite[Chapter 3]{YZ21}. Our notations here are the same as the notations in loc. cit. except that we use the notation $\TPic(U/K)$ instead of the notation $\HPic(U/K)$ in loc. cit. 

Let $K$ be a complete valuation field. By a \emph{variety}, we mean an integral scheme separated of finite type over a base field. For a variety $Y$ over $K$, denote by $Y^\an$ the Berkovich space of $Y$ over $K$ in the sence of \cite{Ber90}.	

By a \emph{Cartier $\ZZ$-divisor} (resp. \emph{$\ZZ$-line bundle}) on a variety $Y$, we mean a Cartier divisor (resp. line bundle) in the usual sense. Denote by $\Div(Y)$ (resp. $\Pic(Y)$) the group of Cartier $\ZZ$-divisors (resp. $\ZZ$-line bundles) on $Y$. By a \emph{Cartier $\QQ$-divisor} (resp. \emph{$\QQ$-line bundle}), we mean an element in $\Div(Y)_\QQ=\Div(Y)\otimes_\ZZ\QQ$ (resp. $\Pic(Y)_\QQ=\Pic(Y)\otimes_\ZZ\QQ$).

Let $Y$ be a variety over $K$. We call $\OD=(D,g)$ an \emph{arithmetic $\ZZ$-divisor} on $Y$, if $D$ is a Cartier $\ZZ-$divisor on $Y$ and $g$ is a continuous Green function of $D$. It means that for any rational function $f$ on a Zariski open subset $V$ of $X$ satisfying $\div(f) = D|_V$, the function $g + \log |f|$ can be extended to a continuous function on $V^\an$. We call $\OL=(L,\|\cdot\|)$ a \emph{metrized $\ZZ$-line bundle} on $Y$, if $L$ is a $\ZZ$-line bundle on $Y$ and $\|\cdot\|$ is a continuous metric of $L$ on $Y^\an$. If $K=\CC$, we need further that the metric of $\OL$ is hermitian, i.e. stable under complex conjugation. Here we use the overline to distinguish an arithmetic divisor (resp. metrized line bundle) with an usual Cartier divisor (resp. line bundle).

Denote by $\HDiv(Y^\an)$ (resp. $\HPic(Y^\an)$) the group of arithmetic $\ZZ$-divisors (resp. metrized $\ZZ$-line bundle) on $Y$. By an \emph{metrized $\QQ$-divisor} (resp. \emph{metrized $\QQ$-line bundle}), we mean an element in $\HDiv(Y^\an)_\QQ=\HDiv(Y^\an)\otimes_\ZZ\QQ$ (resp. $\HPic(Y^\an)_\QQ=\HPic(Y^\an)\otimes_\ZZ\QQ$). In the following, for simplicity, if we are talking about a $\ZZ$-divisor or a $\ZZ$-line bundle, we will omit the ``$\ZZ$'' unless we want to emphasize that.

For a metrized line bundle $\OL=(L,\|\cdot\|)$ and a rational section $s$ of $L$, we can define a natural arithmetic divisor by $\hdiv(s)=(\div(s),-\log\|s\|)$. This process can be done similarly for a metrized $\QQ$-line bundle. Let $\OL=(L,\|\cdot\|)$ be a metrized $\QQ$-line bundle. By a \emph{rational section} $s$ of $L$, it is indeed given by $\frac{1}{n}s_n$. Here $s_n$ is a usual rational section of $L^{\otimes n}$, where $L^{\otimes n}$ is a $\ZZ$-line bundle. In addition, define 
\[
\hdiv(s):=\left(\frac{1}{n}\div(s_n),-\frac{1}{n}\log\|s_n\|\right)\in\HDiv(X)_\QQ.
\]

Define the \emph{Chern current} of an arithmetic $\QQ$-divisor $\OD=(D,g)$ separately in the archimedean case and non-archimedean case. For $K=\CC$, define
\[
c_1(\OD)=\d\dc g+\delta_D,
\]
where 
\[
d=\partial+\bar\partial,\ \dc=\frac{1}{2\pi i}(\partial-\bar\partial),\ \partial=\frac{\partial}{\partial z},\ \bar\partial=\frac{\partial}{\partial \bar{z}}.
\]
For non-archimedean $K$, define
\[
c_1(\OD)=\d'\d'' g+\delta_D,
\]
where the operators $\d'$ and $\d''$ are came up in \cite{CLD12}, which are analogies to $\partial$ and $\bar\partial$ in the complex case. Here $\delta_D$ is the Dirac measure of $D$.

Define the \emph{Chern current} of a metrized $\QQ$-line bundle $\OL$ by
\[
c_1(\OL):=c_1(\hdiv(s)).
\]
Here $s$ is a rational section of $L$ and the definition is independent of $s$. 

For an arithmetic $\QQ$-divisor $\OD=(D,g)$, it is called \emph{effective} if a positive integer multiple of $D$ is an effective Cartier $\ZZ$-divisor and $g\geq0$; called \emph{nef} (resp. \emph{ample}) if a positive integer multiple of $D$ is a nef (resp. ample) Cartier divisor in the usual sense and the Chern current $c_1(\OD)$ is a positive  current on $X^\an$. A metrized line bundle $\OL$ is called \emph{effective} (resp. \emph{nef, ample}) if there is a rational section $s$ of $L$ such that $\hdiv(s)$ is effective (resp. nef, ample).

Let $U$ be a quasi-projective variety over $K$. We call $\CX$ a \emph{projective model} of $U$ over $O_K$, if $\CX$ is an integral projective scheme flat over $O_K$ who contains $U$ as a Zariski open subset. Denote by $X$ the generic fiber $\CX_K$. Then $X$ is a projective variety over $K$ containing $U$ as an open subset. Note that if $K$ is archimedean, we actually have $X=\CX$.

For $\CX,\CY$ two projective models of $U$ over $O_K$, a morphism $f:\CX\to \CY$ is called a \emph{morphism of projective models} if $f|_U=\mathrm{id}_U$. Let $\CX,\CY$ be two projective models of $U$. We say that $\CX$ \emph{dominates} $\CY$ as projective models if there is a morphism of projective models from $\CX$ to $\CY$. For any two projective models $\CX,\CY$ of $U$, there is a projective model $\CZ$ which dominates both $\CX$ and $\CY$. For example, we can take $\CZ$ to be the Zariski closure of $\Delta_U$ in $\CX\times \CY$, where $\Delta_U$ is the diagonal of $U\times U$.

\subsection{Archimedean case}

In this case, $K=\CC$ or $\RR$. Since all constructions on $\CC$ can be moved to $\RR$ by taking the invariant part under the complex conjugation, we may suppose that $K=\CC$.

We call a pair $(X_0,\OD_0)$ a \emph{boundary divisor} of $U$, if $X_0$ is a projective model of $U$ over $O_K$ and $\OD_0$ is an effective arithmetic $\ZZ$-divisor on $X_0$ such that the support $|D_0|=X_0\backslash U$ and $g_0>0$. In the following, we fix a boundary divisor $(X_0,\OD_0)$ of $U$.

An \emph{adelic line bundle} on $U$ is a collection of data $\OL=(L,(X_i,\OL_i,l_i)_{i\geq1})$ satisfying a \emph{Cauchy condition}. Here $X_i$ is a projective model of $U$; $L$ is a $\ZZ$-line bundle on $U$; $\OL_i=(L_i,\|\cdot\|_i)$ is a hermitian $\QQ$-line bundle on $X_i$;  and $l_i:L_i|_U\stackrel{\thicksim}\rightarrow L$ is an ismorphism of line bundles. For any $X_j,X_i$, take a projective model $X_{j,i}$ dominating both $X_j$ and $X_i$. Denote by $p_1:X_{j,i}\to X_j$ and $p_2:X_{j,i}\to X_i$ the morphisms.  For simplicity, denote by $\OL_j\otimes \OL_i^{-1}$ the $\QQ$-line bundle $p_1^*\OL_j\otimes p_2^*\OL_i^{-1}$. Its endowed metric is $\|\cdot\|_j/\|\cdot\|_i$. Define an arithmetic $\QQ$-divisor $\OD_{j,i}$ on $X_{j,i}$ by taking the rational section $l_jl_i^{-1}$ of $L_j\otimes L_i^{-1}$. To be specific,
\[
\OD_{j,i}=\hdiv(l_jl_i^{-1})=(D_{j,i},g_{j,i})=(\div(l_jl_i^{-1}),\|1\|_j/\|1\|_i).
\]
Here $1$ means the rational section of $L_i$ corresponding to the identity section of $\CO_{X_i}$. Then we can state the Cauchy condition of these data as
\[
-\epsilon_i\cdot \OD_0\leq \OD_{j,i}\leq \epsilon_i\cdot \OD_0,\quad \forall j>i\geq1,
\]
where $\epsilon_i$ is a positive constant only depending on $i$ such that $\lim\limits_{i\to\infty}\epsilon_i=0$. Here we use the notation $\OD_0$ to refer its pull-back on $X_{j,i}$. The Cauchy condition is independent of the choice of the projective model $X_{j,i}$. Denote by $\HPic(U/O_K)$ the group of adelic line bundles.

We can also define an \emph{adelic line bundle without metric} $\TL=(L,(X_i,L_i,l_i)_{i\geq1})$ on $U$. Here we use tilde to emphasize that there is no metric data. All the notations and Cauchy conditions are the same as above. The only difference is that we replace hermitian line bundles by line bundles. Denote by $\TPic(U/K)$ the group of adelic line bundles without metrics.

Note that an adelic line bundle $\OL$ (resp. $\TL$) may have different representation of data $\OL=(L,(X_i,\OL_i,l_i)_{i\geq1})$ (resp. $\TL=(L,(X_i,L_i,l_i)_{i\geq1})$). We say an adelic line bundle $\OL$ (resp. $\TL$) is \emph{strongly nef} if there is a representation of data of it such that $\OL_i$ (resp. $L_i$) is nef on $X_i$ for all $i\geq 1$. We say an adelic line bundle is \emph{integrable} if it can be written as the difference of two strongly nef adelic line bundles. 

Denote by $\HPic(U/O_K)_\snef$ (resp. $\TPic(U/K)_\snef$) the subgroup of $\HPic(U/O_K)$ 
\\(resp. $\TPic(U/K)$) consisting of strongly nef adelic line bundles. Denote by $\HPic(U/O_K)_\intt$ (resp. $\TPic(U/K)_\intt$) the subgroup of $\HPic(U/O_K)$ (resp. $\TPic(U/K)$) consisting of integrable adelic line bundles.

\subsection{Non-archimedean case}

In this case, $K$ is a complete valuation field with respect to a non-archimedean absolute value. 

Let $U$ be a quasi-projective variety over $K$. Let $\CX$ be a projective model of $U$ over $O_K$ with generic fiber $X$. Then by the process in \cite{Zha95}, there are natural homomorphisms
\[
\Div(\CX)\longrightarrow\HDiv(X^\an),\quad \Pic(\CX)\longrightarrow\HPic(X^\an).
\]
In addition, for an effective Cartier divisor $\CD\in\Div(\CX)$, its image $\OD\in\HDiv(X^\an)$ is effective. The inverse is true if $\CX$ is normal.

We call a pair $(\CX_0,\CD_0)$ a \emph{boundary divisor} of $U$, if $\CX_0$ is a projective model of $U$ and $\CD_0$ is an effective arithmetic $\ZZ$-divisor on $X_0$ such that the support $|\CD_0|=\CX_0\backslash U$. Note that $\CD_0$ induces an effective arithmetic divisor $\OD_0=(D_0,g_0)$ on $X_0$. In the following, we fix a boundary divisor $(\CX_0,\CD_0)$ of $U$.

The spirit of definitions in this case is similar to the archimedean case. We will only illustrate the definition of $\HPic(U/O_K)$ as an example.

An \emph{adelic line bundle} on $U$ is a collection of data $\OL=(L,(\CX_i,\CL_i,l_i)_{i\geq1})$ satisfying a Cauchy condition. Here $\CX_i$ is a projective model of $U$; $L$ is a $\ZZ$-line bundle on $U$; $\CL_i$ is a $\QQ$-line bundle on $\CX_i$;  and $l_i:\CL_i|_U\stackrel{\thicksim}\rightarrow L$ is an isomorphism of line bundles. For any $\CX_j,\CX_i$, take a projective model $\CX_{j,i}$ dominating both $\CX_j$ and $\CX_i$. Denote by $p_1:\CX_{j,i}\to \CX_j$ and $p_2:\CX_{j,i}\to \CX_i$ the morphisms. For simplicity, denote by $\CL_j\otimes \CL_i^{-1}$ the $\QQ$-line bundle $p_1^*\CL_j\otimes p_2^*\CL_i^{-1}$. Define a Cartier $\QQ$-divisor $\CD_{j,i}$ on $\CX_{j,i}$ by 
\[
\CD_{j,i}=\div(l_jl_i^{-1}).
\]
It induces an arithmetic divisor $\OD_{j,i}=(D_{j,i},g_{j,i})$ on $X_{j,i}$. The Cauchy condition is
\[
-\epsilon_i\cdot \CD_0\leq \CD_{j,i}\leq \epsilon_i\cdot \CD_0,\quad \forall j>i\geq1,
\]
where $\epsilon_i$ is a positive constant only depending on $i$ such that $\lim\limits_{i\to\infty}\epsilon_i=0$. Here we use the notation $\CD_0$ to refer its pull-back on $\CX_{j,i}$. The Cauchy condition is independent of the choice of the projective model $\CX_{j,i}$. Denote by $\HPic(U/O_K)$ the group of adelic line bundles. 

We define $\TPic(U/K)=\HPic(U/O_K)$ in this case. If we talk about an element of $\HPic(U/O_K)$ (resp. $\TPic(U/K)$), we always pay attention to (resp. ignore) the induced metric information on $U^\an$. To distinguish these two cases, for an adelic line bundle $\OL=(L,(\CX_i,\CL_i,l_i)_{i\geq1})$, we also write it as $\OL=(L,(X_i,\OL_i,l_i)_{i\geq1})$ (resp. $\TL=(L,(X_i,L_i,l_i)_{i\geq1})$) when we refer to an element of $\HPic(U/O_K)$ (resp. $\TPic(U/K)$), where $X_i$ is the generic fiber of $\CX_i$ and $\OL_i$ is the image of $\CL_i$ under the natural homomorphism $\Pic(\CX)\to\HPic(X^\an)$. Note that the latter notations have the same form as the archimedean case.

We say an adelic line bundle $\OL$ is \emph{strongly nef} if there is a representation of data of it such that $\CL_i$ is nef on $\CX_i$ for all $i\geq 1$. We say an adelic line bundle is \emph{integrable} if it can be written as the difference of two strongly nef adelic line bundles. 

Denote by $\HPic(U/O_K)_\snef$ the subgroup of $\HPic(U/O_K)$ consisting of strongly nef adelic line bundles. Denote by $\HPic(U/O_K)_\intt$ the subgroup of $\HPic(U/O_K)$ consisting of integrable adelic line bundles. Define $\TPic(U/K)_\snef=\HPic(U/O_K)_\snef$ and $\TPic(U/K)_\intt=\HPic(U/O_K)_\intt$.

\subsection{Natural maps}

Let $U$ be a quasi-projective variety over a complete valuation field $K$. We have recalled several groups in last subsection. In this subsection, we will recall some natural homomorphisms between them.

\begin{itemize}
\item 
The forgetful homomorphism

\[
\begin{array}{r@{\;}c@{\;}l@{\;}c@{\;}l}
\HPic(U/O_K) &\longrightarrow& \TPic(U/K) \\
\OL=(L,(X_i,\OL_i,l_i)_{i\geq1}) &\longmapsto& \TL=(L,(X_i,L_i,l_i)_{i\geq1}).
\end{array}
\]
From the definition, it can be easily seen that it induces 
\[
\HPic(U/O_K)_\snef \longrightarrow \TPic(U/K)_\snef,\quad \HPic(U/O_K)_\intt \longrightarrow \TPic(U/K)_\intt.
\]
	
\item 
A natural homomorphism
\[
\begin{array}{r@{\;}c@{\;}l@{\;}c@{\;}l}
\HPic(U/O_K) &\longrightarrow& \HPic(U^\an) \\
\OL=(L,(X_i,\OL_i,l_i)_{i\geq1}) &\longmapsto& (L,\|\cdot\|).
\end{array}
\]	
The metric $\|\cdot\|$ on $L$ is given by the limit of the metric $\|\cdot\|_i$ on $\OL_i$. The homomorphism is injective by \cite[Proposition 3.6.1]{YZ21}.
\end{itemize}

\subsection{Intersection numbers and measures}

In this subsection, we state two important theorems. The proof of them can be found in \cite[Section 3.6, Section 4.1]{YZ21}.

\begin{theorem}\label{intersection number of adelic line bundles}
	Let $U$ be a quasi-projective variety over a complete valuation field $K$ of dimension $n$. Then there is a well-defined multilinear homomorphism
	\[
	\begin{array}{r@{\;}c@{\;}l@{\;}c@{\;}l}
	\TPic(U/K)_\intt^n &\longrightarrow& \RR \\\\
	(\TL_1,\cdots,\TL_n) &\longmapsto& \TL_1\cdots\TL_n.
	\end{array}
	\]
	In addition, for $(\TL_1,\cdots,\TL_n)\in\TPic(U/K)_\snef^n$ such that each $\TL_i$ is given by the data $\TL_i=(L_i,(X_k,L_{i,k},l_{i,k})_{k\geq1})$ such that $X_k$ is the same for all $i$, we have
	\[
	\TL_1\cdots\TL_n=\lim\limits_{k\to\infty} L_{1,k}\cdots L_{n,k}.
	\]
	Here the right hand side is the usual intersection number on $X_k$.
\end{theorem}

For $\TL_1,\cdots,\TL_n\in\TPic(U/K)_\intt^n$, we call the image $\TL_1\cdots\TL_n$ the \emph{intersection number} of them.

\begin{theorem}\label{measure of adelic line bundles}
Let $U$ be a quasi-projective variety over a complete valuation field $K$ of dimension $n$. Then there is a well-defined multilinear homomorphism
\[
\begin{array}{r@{\;}c@{\;}l@{\;}c@{\;}l}
\HPic(U/O_K)_\intt^n &\longrightarrow& \CM(U^\an) \\\\
(\OL_1,\cdots,\OL_n) &\longmapsto& c_1(\OL_1)\cdots c_1(\OL_n).
\end{array}
\]
Here $\CM(U^\an)$ is the space of measures on $U^\an$. In addition, for $(\OL_1,\cdots,\OL_n)\in\HPic(U/O_K)_\snef^n$ such that each $\TL_i$ is given by the data $\OL_i=(L_i,(X_k,\OL_{i,k},l_{i,k})_{k\geq1})$ such that $X_k$ is the same for all $i$, we have
\[
c_1(\OL_1)\cdots c_1(\OL_n)=\lim\limits_{k\to\infty} c_1(L_{1,k})\cdots c_1(L_{n,k}).
\] 
On the right hand side, if $K=\CC$, the measure is the Monge-Ampère measure in complex analysis; if $K$ is non-archimedean, it is the Chambert-Loir measure. Here the limit is the weak limit, i.e. for each function $f$ supported on a compact subset of $U^\an$,
\[
\int_{U^\an}fc_1(\OL_1)\cdots c_1(\OL_n)=\lim\limits_{k\to\infty}\int_{U^\an}fc_1(L_{1,k})\cdots c_1(L_{n,k}).
\]

\end{theorem}

The measure $c_1(\OL_{1,k})\cdots c_1(\OL_{n,k})$ on the right hand side is positive if each $\OL_{i,k}$ is nef. As a result, for $\OL_1,\cdots,\OL_n\in\HPic(U/O_K)_\snef$, the measure $c_1(\OL_1)\cdots c_1(\OL_n)$ on $U^\an$ is positive.

\section{The proof of the archimedean case}

In this case, $K=\CC$ or $\RR$. It suffices to prove Theorem \ref{main theorem} in the case that $K=\CC$.

\subsection{Reductions} \label{section_redn}

We do some reductions in this subsection. Let $U$ be a quasi-projective variety over $K$ of dimension $n$. Note that the adelic line bundle $\OL_i$ in Theorem \ref{main theorem} is integrable, which means that it is a difference of two strongly nef adelic line bundles. By the multilinearity, it suffices to prove the formula for strongly nef adelic line bundles $\OL_1,\cdots,\OL_n$ on $U$. In addition, there is an equality of polynomials
\[
n!\cdot t_1\cdots t_n=\sum\limits_{I\subset\{1,2,\cdots,n\}}(-1)^{n-\# I}\left(\sum\limits_{i\in I}t_i\right)^n,
\]
which also holds for the intersection numbers $\TL_1\cdots \TL_n$ or the measures $c_1(\OL_1)\cdots c_1(\OL_n)$ by the multilinearity. Hence it suffices to prove that
\[
\int_{U^\an}c_1(\OL)^n=\TL^n
\]
for a single strongly nef adelic line bundle $\OL=(L,(X_i,\OL_i,l_i)_{i\geq1})$ on $U$. In addition, since for each two projective models, there is a third one dominating them, we can assume that $X_j$ dominates $X_i$ for all $j\geq i\geq 0$. In other word, for each $j>i\geq0$, there is a morphism $\phi_{j,i}:X_j\to X_i$ such that $\phi_{j,i}|_U=\mathrm{id}_U$.

We may replace $U$ by its smooth locus. This process shrinks $U$ but doesn't change anything in our theorem. By replacing $X_i$ by its normalization, we may assume that $X_i$ is normal for all $i\geq0$. Since $U$ is smooth, $U$ is stable under this process. Thus $X_i$ is still a projective model of $U$. Note that after the normalization, the domination condition that $X_j$ dominates $X_i$ for all $j> i\geq0$ still holds.

Fix a boundary divisor$(X_0,\OD_0)$, where $\OD_0=(D_0,g_0)$ is effective. We need the following lemma to modify the boundary divisor.

\begin{lemma} \label{finite morphism lemma}
Let $K$ be an infinite field. Let $X$ be a normal projective variety over $K$ of dimension $n$. Then for an effective Cartier divisor $D$ on $X$, there is a finite surjective morphism $f:X\to\PP^n$ such that $f(|D|)$ is of pure dimension $n-1$. Here $|D|$ is the support of $D$.
\end{lemma}

\begin{proof}
	Since $X$ is normal, $D$ induces a nonzero effective Weil divisor, which we still denote by $D$. Write the Weil divisor $D=\sum a_iD_i$, where $D_i$ is a prime Weil divisor. Embed $X$ into $\PP^N$ for some $N$. Take a linear subspace $Z\cong\PP^{N-n-1}$ such that $Z\cap X=\emptyset$. Take an arbitrary linear projective subspace $Y\cong\PP^n$ such that $Y\cap Z=\emptyset$. Let $f:\PP^N\backslash Z\to Y$ be the projection from $Z$ to $Y$. By \cite[Proposition \uppercase\expandafter{\romannumeral2}.7.6]{Mum99}, $f$ is surjective and finite. Then $f(D_i)$ is of pure dimension $n-1$ for all $i$.
\end{proof}

By the lemma, there is a finite surjective morphism $f: X_0\to \PP^n$. Then for some sufficiently large $d$, there exists an effective section $H=\div(s_H)$ of $\CO_{\PP^n}(d)$ such that $H-f_*(D_0)$ is ample. In particular, $f_*(D_0)\subset H$. Here we write $D_0$ for the Weil divisor induced by $D_0$.

Take an arbitrary smooth semi-positive metric of $\CO_{\PP^n}(d)$. For example, take 
\[
\|s(z)\|_{\sm}=\frac{|s(z)|}{\sum\limits_{\alpha_0+\cdots+\alpha_n=d \atop \alpha_i\in\ZZ,\ \alpha_i\geq0 }|z_0|^{\alpha_0}\cdots|z_n|^{\alpha_n}}.
\]
Let $h=-\log |s_H|$. Then $(H,h)=\hdiv(s_H)$, a section of $\overline{\CO}_{\PP^n}(d)=(\CO_{\PP^n}(d),\|\cdot\|_\sm)$. Since the metric is smooth and positive, $h$ is a smooth psh function. In addition, we may change $h$ by a constant such that $f^*h-g_0>0$. Then we may shrink $U$ and replace $(D_0,g_0)$ by $(f^*H,f^*h)=f^*(H,h)$.

In summarize, we may assume that the boundary divisor $(D_0,g_0)=f^*(H,h)$. And it suffices to prove
\[
\int_{U(\CC)}c_1(\OL)^n=\TL^n
\]
for a single strongly nef adelic line bundle $\OL=(L,(X_i,\OL_i,l_i)_{i\geq1})$ on $U$. Under this assumption, the measures $c_1(\OL)^n$ and $c_1(\OL_i)^n$ on $U$ are positive.

\subsection{The proof}

Resume the notations in the last subsection. 

The following lemma constructs a smooth dsh function on $\PP^n(\CC)$. This step is inspired by \cite{DS09,GOV19,GV19}.

\begin{lemma} \label{dsh lemma}
Let $A>0$. Let $(H,h)$ be an effective arithmetic divisor on $\PP^n$ such that $h$ is smooth and psh on $U_H=\PP^n-|H|$. Then there is a smooth function $\chi_A$ on $U_H(\CC)$ such that there exist smooth positive $(1,1)$-forms $\omega^+_A$ and $\omega^-_A$ on $U_H(\CC)$ satisfying

\begin{enumerate}[(1)]
\item 
$0\leq\chi_A\leq1$,	
	
\item 
$\chi_A$ is compactly supported on $\{z\in U_H(\CC)|h(z)\leq A\}$,
	
\item 
$\d\dc\chi_A=\omega^+_A-\omega^-_A$ as differential forms,

\item 
for any closed $(n-1,n-1)$ form $S$ on $\PP^n(\CC)$, we have 
\[
\int_{U_H(\CC)}\omega^+_A\wedge S
=\int_{U_H(\CC)}\omega^-_A\wedge S
=\frac{1}{A}\int_{U_H(\CC)}\d\dc h\wedge S,
\]
where the integration is finite.
\end{enumerate}

\end{lemma}

\begin{proof}
Define $\chi:\RR_{\geq0}\to [0,1]$ by
\[
\chi(x)=
\begin{cases}
\e^{\frac{1}{x-1}+1} &  \text{if $x\in[0,1)$,} \\
0  & \text{if $x\geq1$.}
\end{cases}
\]
Then $\chi(x)$ is a smooth function such that $0\geq\chi'(x)\geq-1$ and $\chi''(x)\geq0$. Define 
\[
\chi_A=\chi\left(\frac{h(z)}{A}\right).
\]
This function satisfies condition (1) and (2).

Note that 
\[
\d\dc\chi_A=\d\dc\chi\left(\frac{h(z)}{A}\right)
=\frac{1}{A}\chi'\left(\frac{h}{A}\right)\d\dc h+\frac{1}{A^2}\chi''\left(\frac{h}{A}\right)\d h\wedge \dc h.	
\]
Here $\d\dc h$ and $\d h\wedge \dc h$ are positive. Take 
\[
\omega^+_A=\d\dc\chi_A+\frac{1}{A}\d\dc h=\frac{1}{A}\left(\chi'+1\right)\d\dc h+\frac{1}{A^2}\chi''\d h\wedge \dc h
\]
\[
\omega^-_A=\frac{1}{A}\d\dc h.
\]
Note that $\chi'> -1$ and $\chi''\geq0$. Then $\omega^+_A$ and $\omega^-_A$ are positive.	

For condition $(4)$, since $\d\dc\chi_A$ is a closed compactly supported differential form on $U_H(\CC)$, by Stokes formula,
\[
\int_{U_H(\CC)}\d\dc\chi_A\wedge S=0.	
\]	
Thus 
\[
\int_{U_H(\CC)}\omega^+_A\wedge S
=\int_{U_H(\CC)}\omega^-_A\wedge S.
\]
\end{proof}

\begin{remark}
Note that the term $\d h\wedge\dc h$ does not have finite mass on the whole $U_H$, while the forms $\omega_A^+$, $\omega_A^-$ do.
\end{remark}

We have a direct corollary.

\begin{corollary} \label{dsh lemma on varity}
	Let $A>0$. Let $X$ be a normal projective variety over $\CC$. Let $(D,g)$ be an effective arithmetic divisor on $X$. Let $U=X-|D|$. Then there are a Zariski open subset $V$ of $U$ and a smooth function $\chi_A$ on $V(\CC)$ such that there exist positive $(1,1)$-forms $\omega^+_A$ and $\omega^-_A$ on $V(\CC)$ satisfying
	
\begin{enumerate}[(1)]
\item 
$0\leq\chi_A\leq1$,	
	
\item 
$\chi_A$ is supported on a compact subset of $\{u\in V(\CC)|g(u)\leq A\}$,
		
\item 
$\d\dc\chi_A=\omega^+_A-\omega^-_A$ as differential forms,
		
\item 
for any closed $(n-1,n-1)$ form $S$ on $X(\CC)$, we have 
\[
\int_{V(\CC)}\omega^+_A\wedge S
=\int_{V(\CC)}\omega^-_A\wedge S
=\frac{C_0}{A}.
\]
Here $C_0=\int_{V(\CC)}f^*(\d\dc h)\wedge S$ is a constant independent of $A$.
\end{enumerate}
\end{corollary}

\begin{proof}
By Lemma \ref{finite morphism lemma}, there exists a finite surjective morphism $f:X\to\PP^n$ such that $f(|D|)$ is of pure dimension $n-1$. Similarly to the reduction step for boundary divisors, there exists an effective arithmetic divisor $(H,h)$ which is a section of $\overline{\CO}_{\PP^n}(d)=(\CO_{\PP^n}(d),\|\cdot\|_\sm)$ such that $D\subset f^*H$, $f^*h>g$ and that $h$ is a smooth psh function on $U_H=\PP^n-|H|$. Resume the function $\chi$ in the proof of Lemma \ref{dsh lemma}. Take 
\[
\chi_A=\chi\left(\frac{h(f(u))}{A}\right),\quad V=f^{-1}U_H.
\]
By an argument similar to Lemma \ref{dsh lemma}, this function satisfies the condition (1), (3) and (4). The constant is
\[
C_0=\int_{V(\CC)}f^*(\d\dc h)\wedge S.
\]	
For condition (2), $\chi_A$ is supported on $\{(f^*h)(u)\leq A\}$, which is a compact subset of $\{g(u)\leq A\}$.
\end{proof}

Let us come back to the setting in Section \ref{section_redn}, where $U$ is a quasi-projective variety over $\CC$ with the boundary divisor $(D_0,g_0)=f^*(H,h)$ on $X_0$, where $(H,h)=\hdiv(s_H)$, a section of $\overline{\CO}_{\PP^n}(d)=(\CO_{\PP^n}(d),\|\cdot\|_\sm)$. Let $\OL=(L,(X_i,\OL_i,l_i)_{i\geq1})$ be a strongly nef adelic line bundle on $U$, which means each $\OL_i$ is nef. Recall that the Cauchy condition is that 
\[
-\epsilon_i\cdot \phi^*_{j,0}\OD_0\leq \OD_{j,i}\leq \epsilon_i\cdot \phi^*_{j,0}\OD_0,\quad \forall j>i\geq1,
\]
where 
\[
\OD_{j,i}=\hdiv(l_jl_i^{-1})=(D_{j,i},g_{j,i})=(\div(l_jl_i^{-1}),\log(\|1\|_j/\|1\|_i)).
\]
Here $\{\epsilon_i\}_{i\geq1}$ is a decreasing sequence of positive rational numbers tending to zero. Let $\chi_A$ be the function in Corollary \ref{dsh lemma on varity} for the pair $(X_0,(D_0,g_0))$. In present situation, we have $V=U=X-f^{-1}|H|$.

\begin{proposition} \label{complex last proposition}
There is a constant $C$ independent of $i,j,A$ such that for all $j> i\geq 1$ and $A>0$, it holds that
\[
\left|\int_{U(\CC)}\chi_Ac_1(\OL_i)^n-\int_{U(\CC)}\chi_Ac_1(\OL_j)^n\right|\leq \epsilon_iC.
\]
\end{proposition}	

\begin{proof}
Denote $U_A=\supp(\chi_A)\subset U$. By Stokes formula,
\begin{align*}
\left|\int_{U(\CC)}\chi_Ac_1(\OL_i)^n-\int_{U(\CC)}\chi_Ac_1(\OL_j)^n\right|
&\leq \sum\limits_{k=0}^{n-1}\left|\int_{U(\CC)}\chi_A\d\dc g_{j,i}\wedge c_1(\OL_i)^k\wedge c_1(\OL_j)^{n-1-k}\right|\\
&=\sum\limits_{k=0}^{n-1}\left|\int_{U_A(\CC)}g_{j,i}\d\dc\chi_A\wedge c_1(\OL_i)^k\wedge c_1(\OL_j)^{n-1-k}\right|.
\end{align*}
By the Cauchy condition, $-\epsilon_i g_0\leq g_{j,i}\leq \epsilon_i g_0$. Thus on $U_A$, $|g_{j,i}|\leq \epsilon_iA$. By Lemma \ref{dsh lemma}, we have  	
\begin{align*}
&\ \ \ \left|\int_{U_A(\CC)}g_{j,i}\d\dc\chi_A\wedge c_1(\OL_i)^k\wedge c_1(\OL_j)^{n-1-k}\right| \\
&\leq \left|\int_{U_A(\CC)}g_{j,i}\omega^+_A\wedge c_1(\OL_i)^k\wedge c_1(\OL_j)^{n-1-k}\right|
 +\left|\int_{U_A(\CC)}g_{j,i}\omega^-_A\wedge c_1(\OL_i)^k\wedge c_1(\OL_j)^{n-1-k}\right|\\
&\leq \epsilon_iA\left(\left|\int_{U(\CC)}\omega^+_A\wedge c_1(\OL_i)^k\wedge c_1(\OL_j)^{n-1-k}\right|+\left|\int_{U(\CC)}\omega^-_A\wedge c_1(\OL_i)^k\wedge c_1(\OL_j)^{n-1-k}\right|\right)\\
&=2\epsilon_i \left|\int_{U(\CC)}\d\dc g_0\wedge c_1(\OL_i)^k\wedge c_1(\OL_j)^{n-1-k}\right|\\
&=2\epsilon_i (D_0\cdot L_i^k\cdot L_j^{n-1-k}).
\end{align*}	
For the first equality, by Corollary \ref{dsh lemma on varity}, we have 
\[
\int_{U(\CC)}\omega^+_A\wedge S
=\int_{U(\CC)}\omega^-_A\wedge S
=\frac{C_{0}}{A},
\quad C_0=\int_{U(\CC)}\d\dc g_0\wedge S.
\]
Here $S=c_1(\OL_i)^k\wedge c_1(\OL_j)^{n-1-k}$.
For the last equality, we use Theorem \ref{integration formula in proj complex case}. 

Note that the intersection numbers $D_0\cdot L_i^k\cdot L_j^{n-1-k}$ are bounded for all $j>i$ since it has the limit $D_0\cdot \widetilde{L}^{n-1}$ when $i\to\infty$. This bound is independent of $i,j$. It completes the proof.	
\end{proof}

\begin{proof}[Proof of Theorem \ref{main theorem} in the archimedean case]
Proposition \ref{complex last proposition} tells us
\[
\left|\int_{U(\CC)}\chi_Ac_1(\OL_i)^n-\int_{U(\CC)}\chi_Ac_1(\OL_j)^n\right|\leq \epsilon_iC,\  \forall j>i\geq1.
\]
Recall that we define the measure $c_1(\OL)^n$ on $U$ by the weak convergence of the measures $c_1(\OL_i)^n$. Note that $\chi_A$ is compactly supported on $U$. Let $j\to\infty$, we get 
\[
\left|\int_{U(\CC)}\chi_Ac_1(\OL_i)^n-\int_{U(\CC)}\chi_Ac_1(\OL)^n\right|\leq \epsilon_iC,\  \forall i\geq1.
\]
Note that the function sequence 
\[\left\{\chi_A(u)\right\}_{A\in\ZZ_+}\] 
is monotone increasing and converges to $1$ pointwise on $U(\CC)$ for $A\to \infty$. Since the measures $c_1(\OL)^n$ and $c_1(\OL_i)^n$ are positive on $U(\CC)$, by Lebesgue's monotone convergence theorem, we have
\[
\left|\int_{U(\CC)}c_1(\OL_i)^n-\int_{U(\CC)}c_1(\OL)^n\right|\leq \epsilon_iC,\  \forall i\geq1.
\]
By Theorem \ref{integration formula in proj complex case},
\[
\left|\int_{U(\CC)}c_1(\OL)^n-(L_i)^n\right|\leq \epsilon_iC,\  \forall i\geq1,
\]
where $(L_i)^n$ is the algebraic intersection number on $X_i$. Finally, let $i\to\infty$. By Theorem \ref{intersection number of adelic line bundles}, we have
\[
\lim\limits_{i\to\infty} (L_i)^n=\TL^n.
\]
Thus
\[
\int_{U(\CC)}c_1(\OL)^n=\TL^n.
\]
It completes the proof.
\end{proof}

\section{The proof of the non-archimedean case}

In this case, $K$ is a non-archimedean complete valuation field. Denote $O_K$ the valuation ring of $K$. We prove Theorem \ref{main theorem} in this case.

The main idea of the non-archimedean case is actually the same as in the archimedean case. The essential difference is that we use different theories of differentials.

\subsection{Differential analysis on Berkovich spaces}

In this subsection, we introduce the differential analysis theory on Berkovich spaces briefly. For the general theory on analytic spaces over a non-archimedean field, we refer to \cite{CLD12}, which is the pioneering paper. For the special case that the analytic space is an analytification of a variety, we also recommend a great reference \cite{Gub16}. Indeed, in this paper, we will only use the special case, which is technically easier than the general case.

There is a \emph{tropicalization map}
\[
\trop \colon (\Gm^n)^\an \longrightarrow \RR^n,\quad x \longmapsto (\log|z_1|_x,\cdots,\log|z_n|_x),
\]
where $z_1,\cdots,z_n$ are coordinates of $\Gm^n$. In fact, by this map, we can identify $\RR^n$ with the canonical skeleton of $(\Gm^n)^\an$ as in \cite[4.4]{Gub16}. We call $\RR^n$ the \emph{tropicalization} of $\Gm^n$. By the Bieri-Groves theorem, for a closed subvariety $Z$ of $\Gm^n$ of dimension $d$, $\trop(Z^\an)$ is a finite union of $d$-dimensional polyhedra in $\RR^n$, which is called the tropicalization of $Z$. The idea is to do analysis on the tropicalization of $Z$. We refer the Bieri-Groves theorem to \cite[Theorem A]{BG84} as the original version. For a translation into tropical language, we refer to \cite[Theorem 2.2.3]{EKL06}.  

For a variety $X$ over $K$, a Zariski open subset $V\subset X$ is called \emph{very affine}, if there is a closed embedding $V\to\Gm^N$ for some $N$. There is a conclusion saying that the very affine open subsets form a topological basis with respect to the Zariski topology of $X$. A function $f$ is called \emph{smooth} on $X^\an$, if  it is locally the pull-back of a smooth superform on $\RR^n$ with respect to some tropicalization map.

In \cite{CLD12}, authors introduce the concept \emph{real differential forms} as an analogy of differential forms in complex analysis. We also have the concept of $(p,q)$-forms in this situation. However, the theory is too enormous to illustrate it here. We will only describe the integraton of a compactly supported smooth differential form of top degree. Let $\omega$ be a top form on $X^\an$. There is a very affine open subset $V$ of $X$ such that $\omega|_V$ is the pull-back of a superform $\eta$ on $\trop(V^\an)$. Then we define the \emph{integration} by 
\[
\int_{X^\an}\omega:=\int_{\trop(V^\an)}\eta.
\]
This definition is independent of the choice of $V$. We can do the integration on $\trop(V^\an)$ since it is a subset of $\RR^n$.

\subsection{The proof}

For the non-archimedean case, we will illustrate the differences to the archimedean case and omit the same parts.

As in the archimedean case, it suffices to prove
\[
\int_{U^\an}c_1(\OL)^n=\TL^n
\]
for a single strongly nef adelic line bundle $\OL=(L,(\CX_i,\CL_i,l_i)_{i\geq1})$ on $U$ such that for each $j>i\geq0$, there is a morphism $\widetilde{\phi}_{j,i}:\CX_j\to \CX_i$ such that $\widetilde{\phi}_{j,i}|_U=\mathrm{id}_U$. Denote by $X_i=\CX_i\times_{\Spec O_K}\Spec K$ the generic fiber of $\CX_i$. Denote by $\phi_{j,i}:X_j\to X_i$ the induced morphism of generic fibers.

We may also assume that $U$ is normal and that $\CX_i$ is normal for all $i$. Note that each Cartier divisor $\CD_i$ on $\CX_i$ naturally defines an arithmetic divisor $(D_i,g_i)$ on $X_i$. The Cauchy condition is that
\[
-\epsilon_i\cdot \widetilde{\phi}^*_{j,0}\CD_0\leq \CD_{j,i}\leq \epsilon_i\cdot \widetilde{\phi}^*_{j,0}\CD_0,\quad \forall j>i\geq1.
\]
Here $\{\epsilon_i\}_{i\geq1}$ is a decreasing sequence of positive rational numbers tending to zero. It implies that
\[
-\epsilon_i\cdot \phi^*_{j,i}D_0\leq D_{j,i}\leq \epsilon_i\cdot \phi^*_{j,0}D_0,
\quad -\epsilon_ig_0\leq g_{j,i}\leq \epsilon_ig_0,
\quad \forall j>i\geq1.
\]
Thus we have inequalities of the same form as the archimedean case, which help us do the proof in the same way as the archimedean case.

Note that Lemma \ref{finite morphism lemma} still holds in the non-archimedean cases. By the lemma, there is a finite surjective morphism $f: X_0\to \PP^n_K$. Then for some sufficiently large $d$, there exists an effective section $H=\div(s_H)$ of $\CO_{\PP^n_K}(d)$ such that $H-f_*(D_0)$ is ample. In particular, $f_*(D_0)\subset H$. 

As in the complex case, take an arbitrary smooth psh metric $\|\cdot\|_{\sm}$ of $\CO_{\PP^n_K}(d)$. By \cite[Corollary 6.3.3]{CLD12}, such a metric does exist. Here a metric is said to be \emph{smooth} in the sense of \cite[Definition 6.2.4]{CLD12}, and is said to be \emph{psh} in the sense of \cite[6.3.1]{CLD12}. Then we can take $(H,h)=\hdiv(s_H)$ as in the complex case and the differential form $\d'\d'' h$ is positive. In addition, we may change $h$ by a constant such that $f^*h-g_0>0$. Then we may shrink $U$ and replace $(D_0,g_0)$ by $(f^*H,f^*h)=f^*(H,h)$.

We will state the lemma and its corollary without detailed proof. Its proof is exactly the same as in the complex case.
	
\begin{lemma} \label{dsh lemma in NA}
Let $A>0$. Let $(H,h)$ be an effective arithmetic divisor on $\PP^n_K$ such that $h$ is smooth and psh on $U_H=\PP^n_K-|H|$. Then there is a smooth function $\chi_A$ on $U_H^\an$ such that there exist positive $(1,1)$-forms $\omega^+_A$ and $\omega^-_A$ on $U_H^\an$ satisfying

\begin{enumerate}[(1)]
	\item 
	$0\leq\chi_A\leq1$,	
	
	\item 
	$\chi_A$ is compactly supported on $\{z\in U_H^\an|h(z)\leq A\}$,
	
	\item 
	$\d'\d''\chi_A=\omega^+_A-\omega^-_A$ as differential forms,
	
	\item 
	for any closed $(n-1,n-1)$ smooth differential form $S$ on $(\PP^n_K)^\an$, we have 
	\[
	\int_{U_H^\an}\omega^+_A\wedge S
	=\int_{U_H^\an}\omega^-_A\wedge S
	=\frac{1}{A}\int_{U_H^\an}\d'\d'' h\wedge S,
	\]
	where the integration is finite.
\end{enumerate}
	
\end{lemma}

\begin{proof}
	Define $\chi:\RR_{\geq0}\to [0,1]$ by
	\[
	\chi(x)=
	\begin{cases}
	\e^{\frac{1}{x-1}+1} &  \text{if $x\in[0,1)$,} \\
	0  & \text{if $x\geq1$.}
	\end{cases}
	\]
	Then we take
	\[
	\chi_A=\chi\left(\frac{h(x)}{A}\right),\quad x\in (U_H)^\an
	\]
	and 
	\[
	\omega_A^+=\d'\d''\chi_A+\frac{1}{A}\d'\d''h,\quad \omega_A^-=\frac{1}{A}\d'\d''h.
	\]
	This function satisfies desired conditions. The proof is the same as in Lemma \ref{dsh lemma}. The argument relies on Green's formula in the non-archimedean setting, which is proved in \cite[Theorem 3.12.2]{CLD12}.
\end{proof}	
	
\begin{corollary} \label{dsh lemma on varity in NA}
	Let $A>0$. Let $X$ be a normal projective variety over $K$. Let $(D,g)$ be an effective arithmetic divisor on $X$. Let $U=X-|D|$. Then there are a Zariski open subset $V$ of $U$ and a smooth function $\chi_A$ on $V^\an$ such that there exist positive $(1,1)$-forms $\omega^+_A$ and $\omega^-_A$ on $V^\an$ satisfying
	
	\begin{enumerate}[(1)]
		\item 
		$0\leq\chi_A\leq1$,	
		
		\item 
		$\chi_A$ is supported on a compact subset of $\{u\in V^\an|g(u)\leq A\}$,
		
		\item 
		$\d'\d''\chi_A=\omega^+_A-\omega^-_A$ as differential forms,
		
		\item 
		for any closed $(n-1,n-1)$ smooth differential form $S$ on $X^\an$, we have 
		\[
		\int_{V^\an}\omega^+_A\wedge S
		=\int_{V^\an}\omega^-_A\wedge S
		=\frac{C_0}{A}.
		\]
		Here $C_0=\int_{V^\an}f^*(\d'\d'' h)\wedge S$ is a constant independent of $A$.
	\end{enumerate}
\end{corollary}	

\begin{proof}
	As in the proof of Corollary \ref{dsh lemma on varity}, there exists a finite surjective morphism $f:X\to\PP^n_K$ and an effective arithemetic divisor $(H,h)$ which is a section of $\overline{\CO}_{\PP^n_K}(d)=(\CO_{\PP^n_K}(d),\|\cdot\|_\sm)$ such that $D\subset f^*H$, $f^*h>g$ and that $h$ is a smooth psh function on $U_H=\PP^n_K-|H|$. Resume the function $\chi$ in the proof of Lemma \ref{dsh lemma in NA}. Take 
	\[
	\chi_A=\chi\left(\frac{h(f(u))}{A}\right),\quad V=f^*U_H.
	\] 
	Then by an argument similar to Lemma \ref{dsh lemma in NA}, this function satisfies the conditions. The constant is
	\[
	C_0=\int_{V^\an}f^*(\d'\d'' h)\wedge S.
	\]
\end{proof}

By Corollary \ref{dsh lemma on varity in NA}, we can prove the non-archimedean version of Proposition \ref{complex last proposition}. 

\begin{proposition} \label{NA last proposition}
	There is a constant $C$ independent of $i,j,A$ such that for all $j> i\geq 1$ and $A>0$, it holds that
	\[
	\left|\int_{U^\an}\chi_Ac_1(\OL_i)^n-\int_{U^\an}\chi_Ac_1(\OL_j)^n\right|\leq \epsilon_iC.
	\]
\end{proposition}	

\begin{proof}
	Denote $U_A=\supp(\chi_A)\subset U$. As in the complex case, by Stokes formula,
	\begin{align*}
	\left|\int_{U^\an}\chi_Ac_1(\OL_i)^n-\int_{U^\an}\chi_Ac_1(\OL_j)^n\right|
	&\leq \sum\limits_{k=0}^{n-1}\left|\int_{U^\an}\chi_A\d\dc g_{j,i}\wedge c_1(\OL_i)^k\wedge c_1(\OL_j)^{n-1-k}\right|\\
	&=\sum\limits_{k=0}^{n-1}\left|\int_{U_A^\an}g_{j,i}\d\dc\chi_A\wedge c_1(\OL_i)^k\wedge c_1(\OL_j)^{n-1-k}\right|\\
	&\leq 2\epsilon_iA\sum\limits_{k=0}^{n-1}\left|\int_{U^\an}\omega^+_A\wedge c_1(\OL_i)^k\wedge c_1(\OL_j)^{n-1-k}\right|\\
	&=2\epsilon_i \left|\int_{U^\an}\d\dc g_0\wedge c_1(\OL_i)^k\wedge c_1(\OL_j)^{n-1-k}\right|\\
	&=2\epsilon_i (D_0\cdot L_i^k\cdot L_j^{n-1-k}).
	\end{align*}	
	For the last equality, we use a non-archimedean version of Theorem \ref{integration formula in proj complex case}. For example, we can use \cite[Corollary 6.4.4]{CLD12}. Note that the function $g_{j,i}$ may not be smooth. To handle this, we may use the theory of approximable psh functions developed in \cite{CLD12} or the theory of delta-forms in \cite{GK17}.
	
	Note that the intersection numbers $D_0\cdot L_i^k\cdot L_j^{n-1-k}$ are bounded for all $j>i$ since it has the limit $D_0\cdot \widetilde{L}^{n-1}$ when $i\to\infty$. This bound is independent of $i,j$. It completes the proof.	
\end{proof}	
	
	According to Proposition \ref{NA last proposition}, we can prove the main theorem in the non-archimedean case in the same way as in the archimedean case.
\begin{proof}[Proof of Theorem \ref{main theorem} in the non-archimedean case]
	Proposition \ref{NA last proposition} tells us
	\[
	\left|\int_{U^\an}\chi_Ac_1(\OL_i)^n-\int_{U^\an}\chi_Ac_1(\OL_j)^n\right|\leq \epsilon_iC,\  \forall j>i\geq1.
	\]
	By the weak convergence of the measures in Theorem \ref{measure of adelic line bundles}, for $j\to\infty$, we have 
	\[
	\left|\int_{U^\an}\chi_Ac_1(\OL_i)^n-\int_{U^\an}\chi_Ac_1(\OL)^n\right|\leq \epsilon_iC,\  \forall i\geq1.
	\]
	Let $A\to \infty$, by Lebesgue's monotone convergence theorem, we have
	\[
	\left|\int_{U^\an}c_1(\OL_i)^n-\int_{U^\an}c_1(\OL)^n\right|\leq \epsilon_iC,\  \forall i\geq1.
	\]
	By \cite[Corollary 6.4.4]{CLD12}, a non-archimedean version of Theorem \ref{integration formula in proj complex case}, we have
	\[
	\left|\int_{U^\an}c_1(\OL)^n-(L_i)^n\right|\leq \epsilon_iC,\  \forall i\geq1,
	\]
	Finally, let $i\to\infty$. By Theorem \ref{intersection number of adelic line bundles}, we have
	\[
	\lim\limits_{i\to\infty} (L_i)^n=\TL^n.
	\]
	Thus
	\[
	\int_{U^\an}c_1(\OL)^n=\TL^n.
	\]
\end{proof}

\

\noindent \small{School of Mathematical Science, Peking University, Beijing 100871, China}

\noindent \small{\it Email: guoruoyi@pku.edu.cn}

\end{document}